\documentclass[12pt]{amsart}
\usepackage[OT4]{fontenc}
\usepackage{graphics,graphicx}
\usepackage{amsmath}
\usepackage{amssymb,amsfonts}
\usepackage {amsthm}
\usepackage{xcolor}

\usepackage{enumerate}
\usepackage{hyperref}
\usepackage[T2A]{fontenc}
\usepackage[cp1251]{inputenc}
\usepackage[all]{xy}

\newcommand\myfootnote[1]{
\renewcommand{\thefootnote}{}
\footnotetext{#1}
\def\thefootnote{\@arabic\c@footnote}
}

\makeatletter
\renewcommand{\subsection}{\@startsection{subsection}{2}{0mm}{-\baselineskip}{-5pt}{\it \bf}}
\makeatother

\newtheorem{theorem}{Theorem}
\newtheorem{lemma}{Lemma}
\newtheorem{corollary}{Corollary}

\title{ON THE LIE STRUCTURE OF LOCALLY MATRIX ALGEBRAS}
\author{OKSANA BEZUSHCHAK}

\begin{document}

\maketitle

\address{Faculty of Mechanics and Mathematics,
Taras Shevchenko National University of Kyiv\\ Volodymyrska 60, Kyiv 01033, Ukraine\\
\email{bezusch@univ.kiev.ua}}

\begin{abstract}
Let $A$ be a unital  locally matrix algebra over a field $\mathbb{F}$ of characteristic different from $2.$ We find necessary and sufficient condition for the Lie algebra $A\diagup\mathbb{F}\cdot 1$ to be simple and for the Lie algebra of derivations $\text{Der}(A)$ to be topologically simple. The condition depends on the Steinitz number of $A$ only.


	
	\end{abstract}

\section*{Introduction}

Let $F$ be a ground  field of characteristic different from $2$ and let $ \mathbb{N} $ be the set of all positive integers. Recall that an associative $F$--algebra  $A$ is called  a \emph{locally matrix algebra} (see \cite{Koethe,Kurosh}) if for an arbitrary finite subset of $A$ there exists a subalgebra   $B\subset A$ containing  this subset and such that  $B\cong M_n(F)$ for some $n\in \mathbb{N}.$ We call a locally matrix algebra \emph{unital} if it contains unit $1.$

Let $ \mathbb{P} $ be the set of all primes. An infinite formal product of the form
\begin{equation}\label{1}
s \ = \ \prod_{p\in \mathbb{P}} p^{r_p} \ , \quad \text{where} \quad    r_p\in  \mathbb{N} \cup \{0,\infty\} \quad \text{for all} \quad p\in \mathbb{P},\end{equation}
is called {\it Steinitz number}; see \cite{ST}. Denote by symbol $ \mathbb{SN} $ the set of all Steinitz numbers. Let
$$ s_1\ = \ \prod_{p\in \mathbb{P}} p^{r_p} \, , \  \ s_2 \ = \ \prod_{p\in \mathbb{P}} p^{k_p} \ \in \mathbb{SN}.   $$
Then
$$ s_1 \cdot s_2 \ = \ \prod_{p\in \mathbb{P}} p^{r_p+k_p}  \ ,  \quad \text{where} \quad k_p \in  \mathbb{N} \cup \{0,\infty\}, $$ and   $ t+\infty=\infty+t=\infty+\infty=\infty $ for all $ t \in \mathbb{N}.$

Let $A$ be a countable--dimensional unital locally matrix algebra. In \cite{Glimm}, J.G.Glimm defined the Steinitz number $\mathbf{st}(A)$ of the algebra $A$ and proved that the algebra $A$ is uniquely determined by $\mathbf{st}(A)$.

In \cite{BezOl}, we extended Glimm's definition to  unital locally matrix algebras of arbitrary dimensions. For a unital locally matrix algebra $A$ denote by  $D(A)$  the  set of all numbers $n\in \mathbb{N}$ such that  there exists  a subalgebra  $A'\subseteq A$, $ 1 \in A'$, and $A' \cong  M_n(F).$ The {\it Steinitz number} $\mathbf{st}(A)$ \emph{of the algebra} $A$ is the  least common multiple of the set $D(A).$ It turned out that a unital locally matrix algebra $A$ of dimension  $> \aleph_0 $ is no longer determined by its Steinitz number $\mathbf{st}(A)$; see \cite{BezOl,BezOl_2}.

An associative algebra $A$ gives rise to the Lie algebra $$A^{(-)} \ = \ (A, \ [a,b]=ab-ba).$$ Along with the Lie algebra $A^{(-)}$ we will consider its square $[A,A].$ Let $Z(A)$ denote the center of the associative algebra $A.$ In \cite{SkolemNoether_1}, I.~N.~Herstein showed that if $A$ is a simple associative algebra then the Lie algebra
\begin{equation}\label{equation1}
[A,A] \, \diagup \, \raisebox{-2pt}{$Z(A)\cap [A,A]$}
\end{equation}
is simple. Since a locally matrix algebra $A$ is simple, it follows that the Lie algebra  (\ref{equation1}) is simple.

Let  $M_n(\mathbb{F})$ be a matrix algebra, $$\mathfrak{gl}(n) \ = \ \big(M_n(\mathbb{F})\big)^{(-)}, \quad  Z(M_n(\mathbb{F})) \ = \ \mathbb{F}\cdot 1.$$ The Lie algebra  $$\mathfrak{pgl}(n)\ =\ \mathfrak{gl}(n)\, \diagup \, \raisebox{-2pt}{$\mathbb{F}\cdot 1$}$$ is simple unless $p=\text{char}\ \mathbb{F}>0$ and  $p$ divides  $n;$ see  \cite{Seligman}. We will show that for an infinite--dimensional unital locally matrix algebra $A$ simplicity of the Lie algebra
\begin{equation}\label{equation2}
A^{(-)}\, \diagup \, \raisebox{-2pt}{$\mathbb{F}\cdot 1$}\end{equation}  depends only on the Steinitz number $\mathbf{st}(A).$

For a Steinitz number (\ref{1}) denote $\nu_p(s)=k_p.$

\begin{theorem}\label{theorem1} The Lie algebra $(\ref{equation2})$ is simple if and only if \begin{equation}\label{Us_1}
                                 \emph{char} \ \mathbb{F}   = 0  \end{equation}  or
\begin{equation*}  \emph{char} \ \mathbb{F}  =  p  > 0  \quad \text{and} \quad  \nu_p  (  \mathbf{st}(A)  ) =  0 \quad  \text{or} \quad  \infty.\end{equation*} \end{theorem}

Recall that a linear transformation $d:A\rightarrow A$ of an algebra $A$ is called a \emph{derivation} if  $$d(ab) \ = \ d(a)\cdot b \ + \  a \cdot d(b)$$ for arbitrary elements $a,\, b\in A.$ The vector space $\text{Der}(A) $ of all derivations of an algebra $A$ is a Lie algebra with respect to commutation; see \cite{Jacobson_2}. If $A$ is an associative algebra then for an arbitrary element $a\in A$ the operator  $$\text{ad}(a): \ A \ \rightarrow \ A,  \quad x \ \mapsto \ [a,x],$$ is an {\it inner derivation}. The subspace $\text{Inder}(A)  =  \{ \, \text{ad}(a) \, | \, a\in A \ \}$ of all  inner derivations is an ideal of the Lie algebra  $\text{Der} (A).$

Let $X$ be an arbitrary set. The set $\text{Map} (X,X)$ of all mappings $X \rightarrow X$ is equipped with Tykhonoff topology; see \cite{Willard}. The subspace of all derivations $\text{Der}(A)$ of an algebra $A$ is closed in  $\text{Map} (A,A)$ in Tykhonoff topology. It makes the Lie algebra $\text{Der}(A)$ a topological algebra.

\begin{theorem}\label{theorem3} Let $A$ be a unital locally matrix algebra. Then
\begin{enumerate}
\item[$(1)$] the Lie algebra $[\emph{\text{Der}}(A),\emph{\text{Der}}(A)]$ is topologically simple;
\item[$(2)$] the Lie algebra $\emph{\text{Der}}(A)$ is topologically simple if and  only if $\emph{char} \ \mathbb{F}   = 0 $ or $\emph{char} \ \mathbb{F}  =  p  > 0 $ and $ \nu_p  (  \mathbf{st}(A)  ) =  0 $ or $ \infty.$ \end{enumerate} \end{theorem}

 \section{Proof of the Theorem \ref{theorem1}}

\begin{proof} Let $A$ be a unital locally matrix $\mathbb{F}$--algebra. We will show that $$A\ = \ [A,A] + \mathbb{F}\cdot 1 \quad \text{if and only if} \quad \text{char} \ \mathbb{F}   = 0$$ $$ \text{or} \quad \text{char} \ \mathbb{F}   = p>0 \quad \text{and} \quad \nu_p(\mathbf{st}(A))=0 \text{ or } \infty.$$

Consider a matrix algebra $M_n(\mathbb{F}).$ Then $$\big[M_n(\mathbb{F}),M_n(\mathbb{F})\big] \ = \ \big\{\, a \in M_n(\mathbb{F}) \, \big| \, \text{tr}(a)=0 \,\big\} \ = \ \mathfrak{sl}(n).$$ If
$\text{char} \ \mathbb{F}   = 0$  or  $\text{char} \ \mathbb{F}   = p>0$  and $p$ does not divide $n$ then $\text{tr}(1)=n \neq 0,$ and therefore $$ a =\Big( a-\frac{1}{n} \ \text{tr}(a)\cdot 1 \Big) \ + \ \frac{1}{n} \ \text{tr}(a)\cdot 1  \ \in \ s\mathfrak{l}(n)+ \mathbb{F} \cdot 1 $$ for an arbitrary element $a\in A.$ It implies that $$M_n(\mathbb{F}) \ = \ \big[M_n(\mathbb{F}),M_n(\mathbb{F})\big] \ + \ \mathbb{F} \cdot 1.$$ If $p$ divides $n$ then $\text{tr}(1)=0,$ hence $$\big[M_n(\mathbb{F}),M_n(\mathbb{F})\big] \ + \ \mathbb{F} \cdot 1  \ = \ \mathfrak{sl}(n) \ \neq \ M_n(\mathbb{F}).$$ If $\text{char} \ \mathbb{F}   = 0 $ or $\text{char} \ \mathbb{F}  =  p  > 0 $ and $p$ does not divide $\mathbf{st}(A)$ then for an arbitrary matrix subalgebra $1\in A_1 \subset A,$ $A_1\cong M_n(\mathbb{F})$ the characteristic $p$ does not divide $n.$ Hence $$A_1\ = \ [A_1,A_1] \ + \ \mathbb{F}\cdot 1.$$ So, $A=[A,A] + \mathbb{F}\cdot 1.$

Suppose now that $p^{\infty}$ divides $\mathbf{st}(A).$ Consider a matrix subalgebra $$1\in A_1 \subset A, \quad A_1\cong M_n(\mathbb{F}).$$ The number $n$ divides $\mathbf{st}(A).$ Since $p^{\infty}$ divides $\mathbf{st}(A)$ it follows that $pn$ also divides $\mathbf{st}(A).$ Hence, there exists a subalgebra $A_2\subset A$ such that $$ A_1 \subset A_2, \quad A_2\cong M_m(\mathbb{F})\quad \text{and} \quad p \quad \text{divides} \quad m/n.$$ Let $C$ be the centralizer of the subalgebra $A_1$ in $A_2.$ We have $$A_2 \ = \ A_1 \ \otimes_{\mathbb{F}} \ C, \quad  C  \cong  M_{m/n}(\mathbb{F}) $$ (see \cite{Drozd_Kirichenko,Jacobson_3}). For arbitrary elements $a\in A_1,$ $b\in C$ we have $$ \text{tr}_{A_2}  (a\, b) \ = \ \text{tr}_{A_1} (a) \ \cdot \ \text{tr}_{C} (b) ,$$ where $\text{tr}_{A_2},$ $\text{tr}_{A_1},$ $\text{tr}_{C}$ are  traces in the subalgebras $A_2,$ $A_1,$ $C,$ respectively. We have $$ \text{tr}_{C}(1) \  = \ \frac{m}{n} \ = \ 0.$$ Hence  $$ \text{tr}_{A_2}  (A_1 \otimes 1) \ = \ \text{tr}_{A_1}  (A_1) \ \cdot \ \text{tr}_{C}  (1) \ = \ \{0\},$$ and therefore  $A_1 \subseteq [A_2, A_2].$ We showed that if $p^{\infty}$ divides $\mathbf{st}(A)$ then $A=[A,A].$

Suppose now that $\nu_p(\mathbf{st}(A))=k,$ $1\leq k < \infty.$ Consider a subalgebra  $$ 1\in A_1 \subset A, \quad A_1\cong M_{p^k}(\mathbb{F}).$$ Choose an element $a\in A_1$ such that $\text{tr}_{A_1} (a) \neq 0.$ We claim that
\begin{equation}\label{equation4}
a\not\in [A,A]+ \mathbb{F}\cdot 1.
\end{equation}
Indeed, if the element $a$ lies in the right hand side then there exists a subalgebra $A_2 \subset A$ such that $$A_1\subset A_2, \quad A_2\cong M_n(\mathbb{F})\quad \text{and} \quad a\in [A_2, A_2] + \mathbb{F}\cdot 1.$$ Since $p$ divides $n$ it follows that $\text{tr}_{A_2}  (1) = 0,$ hence $\text{tr}_{A_2}  (a) = 0.$

As above, let $C$ be the centralizer of the subalgebra $A_1$ in $A_2,$ so that $A_2 = A_1  \otimes_{\mathbb{F}}  C. $ The algebra $C$ is isomorphic to the matrix algebra $M_m(\mathbb{F}),$ where $m = n/ p^k.$ The number  $m$ is coprime with $p,$ hence $\text{tr}_{C}  (1) = m \neq 0.$ Now, $$ \text{tr}_{A_2}(a) \ = \ \text{tr}_{A_2}  (a \otimes 1) \ = \ \text{tr}_{A_1}  (a) \ \cdot \ \text{tr}_{C}  (1) \ \neq \ 0.$$ This contradiction completes the proof of the claim (\ref{equation4}).

If $A=[A,A]+ \mathbb{F} \cdot 1$ then $$A^{(-)}\, \diagup \, \raisebox{-2pt}{$\mathbb{F}\cdot 1$} \ \cong \ [A,A]\, \diagup \, \raisebox{-2pt}{$[A,A] \cap \mathbb{F}\cdot 1$}.$$ In this case, the Lie algebra $A^{(-)}\, \diagup \, \raisebox{-2pt}{$\mathbb{F}\cdot 1$}$ is simple by I.~N.~Herstein's Theorem \cite{SkolemNoether_1}. If  $A=[A,A]+ \mathbb{F} \cdot 1$ is a proper subspace of $A$ then
$$[A,A]+\mathbb{F} \cdot 1 \, \diagup \, \raisebox{-2pt}{$\mathbb{F}\cdot 1$}$$ is a proper  ideal in the Lie algebra $A^{(-)}\, \diagup \, \raisebox{-2pt}{$\mathbb{F}\cdot 1$}.$ This completes the proof of Theorem \ref{theorem1}. \end{proof}

Consider the homomorphism $$ \varphi:A^{(-)} \rightarrow \text{Inder}(A), \quad \varphi(a) = \text{ad}(a), \quad a\in A.$$ Since $\text{Ker}\, \varphi=Z(A) = \mathbb{F}\cdot 1$ it follows that $$ \text{Inder}(A) \ \cong \ A^{(-)} \, \diagup \, \raisebox{-2pt}{$\mathbb{F}\cdot 1$}.$$
\begin{corollary} \begin{enumerate}
                    \item[$(1)$] The Lie algebra $[\emph{\text{Inder}}(A),\emph{\text{Inder}}(A)]$ is simple.
                    \item[$(2)$] The Lie algebra $\emph{\text{Inder}}(A)$ is simple if and only if $\emph{char} \ \mathbb{F}   = 0$   or $ \emph{char} \ \mathbb{F}  =  p  > 0  $ and $ \nu_p  (  \mathbf{st}(A)  ) =  0  $ or $ \infty.$
                  \end{enumerate}
\end{corollary}
\begin{proof} The Lie algebra $$[\text{\text{Inder}}(A),\text{\text{Inder}}(A)] \ \cong \ [A,A] \, \diagup \, \raisebox{-2pt}{$[A,A]\cap \mathbb{F}\cdot 1$}$$ is simple by I.~N.~Herstein's Theorem \cite{SkolemNoether_1}. The part $(2)$ immediately follows from Theorem \ref{theorem1}. \end{proof}

\section{Proof of the Theorem \ref{theorem3}}

\begin{lemma}\label{lemma1} Let $A$ be an infinite--dimensional locally matrix algebra. Let $d\in \emph{\text{Der}}(A)$  and suppose that $d([A,A])$ lies in the center of the algebra $A.$ Then $d=0.$ \end{lemma}
\begin{proof} Let  $Z$ be the center of $A.$ If $A$ is not unital then $Z=\{0\}.$ If $A$ is unital then $Z=\mathbb{F} \cdot 1.$ Consider a subalgebra $A_1 \subset A$ such that $A_1 \cong M_n(\mathbb{F})$ for some  $n\geq 4,$ and let $\varphi:M_n(\mathbb{F}) \rightarrow A_1$ be an isomorphism.   An arbitrary matrix unit  $e_{ij},$ $1 \leq i \neq j \leq n,$ lies in $[M_n(\mathbb{F}),M_n(\mathbb{F})].$ Choose distinct indices  $1 \leq i , j , s, t \leq n.$ Then $e_{ij}=[e_{is}, e_{sj}].$ Hence $$d\big(\varphi(e_{ij})\big) \ \in \ Z \varphi(e_{is}) \ + \ Z \varphi( e_{sj}).$$  On the other hand, $e_{ij}=[e_{it}, e_{tj}],$ which implies $$d\big(\varphi(e_{ij})\big) \ \in \ Z \varphi(e_{it}) \ + \ Z \varphi( e_{tj}).$$ Hence $d(\varphi(e_{ij}))=0.$ The algebra $M_n(\mathbb{F})$  is generated by matrix units $e_{ij},$ $1\leq i \neq j \leq n. $ Hence $d( \varphi(M_n(\mathbb{F}))  )=\{0\},$ and therefore $d(A)=\{0\}.$ This completes the proof of the Lemma.  \end{proof}

In \cite{Bez5}, we showed that for an arbitrary locally matrix algebra $A$ the ideal $\text{Inder}(A)$  is dense in the Lie algebra $\text{Der}(A)$ in the Tykhonoff topology.

\begin{proof}[Proof of Theorem $\ref{theorem3}$] $(1)$ Let $I$ be a nonzero closed ideal of the Lie algebra $[\text{Der}(A), \text{Der}(A)].$  Choose a nonzero element $d\in I.$ For an arbitrary element $a\in [A,A]$ we have $$ [d,\text{ad}(a)] \ = \ \text{ad}\big(d(a)\big) \ \in \ [\text{Inder}(A), \text{Inder}(A)].$$ By Lemma \ref{lemma1}, we can choose an element $a\in [A,A]$ so that $d(a)\neq 0.$ Hence $$ I \ \cap \ [\text{Inder}(A),\text{Inder}(A)] \ \neq \ \{0\}.$$ Since the Lie algebra  $[\text{Inder}(A),\text{Inder}(A)] $ is simple it follows that $$[\text{Inder}(A),\text{Inder}(A)] \subseteq I.$$ We have mentioned above that $\text{Inder}(A)$ is dense in the Lie algebra  $\text{Der}(A) $ in the Tykhonoff topology; see \cite{Bez5}. Hence, $[\text{Inder}(A),\text{Inder}(A)] $ is dense in the Lie algebra $[\text{Der}(A),\text{Der}(A)].$  Since the ideal $I$ is closed we conclude that $I=[\text{Inder}(A),\text{Inder}(A)] .$ This completes the proof of Theorem \ref{theorem3}(1).

\vspace{12pt}

$(2)$ Let $I$ be a nonzero closed ideal of the Lie algebra $\text{Der}(A).$  Choose a nonzero derivation  $d\in I.$ By Lemma \ref{lemma1}, there exists an element $a\in A$ such that $d(a)$ does not lie in $\mathbb{F} \cdot 1,$ hence $$ 0 \ \neq \ \text{ad}\big(d(a)\big) \ = \ [d,\text{ad}(a)] \ \in \ I \cap \ \text{Inder}(A).$$

Suppose that $\text{char} \ \mathbb{F} = 0$ or $\text{char} \ \mathbb{F} = p> 0$ and $\nu_p(\mathbf{st}(A))=0$ or $\infty.$  Then the Lie algebra  $\text{Inder}(A)$ is simple, and therefore $\text{Inder}(A)\subseteq I.$ Since $\text{Inder}(A)$ is dense in $\text{Der}(A)$ (see \cite{Bez5}) and the ideal $I$ is closed it follows that $I=\text{Der}(A).$

Now suppose that $\nu_p(\mathbf{st}(A))=k,$ $1\leq k < \infty.$ There exists a subalgebra $A_1$ in $A$ such that $$ 1\in A_1 \quad \text{and} \quad A_1 \cong M_{p^k}(\mathbb{F}).$$ Choose an element $a\in A_1$ such that $\text{tr}_{A_1}(a)\neq 0.$ We will show that the inner derivation $\text{ad}(a)$ does not lie in the closure $$ \overline{[\text{Der}(A),\text{Der}(A)]} \ = \ \overline{[\text{Inder}(A),\text{Inder}(A)]},$$ and therefore $\overline{[\text{Inder}(A),\text{Inder}(A)]}$ is a proper closed ideal in the Lie algebra $\text{Der}(A).$  If $\text{ad}(a)$ lies in the closure of $ [\text{Inder}(A),\text{Inder}(A)]$  then, by the definition of the Tykhonoff topology, there exist elements $a_i, b_i \in A,$ $1\leq i \leq n,$ such that $$ \Big( \, \text{ad}(a) \ - \ \sum_{i=1}^{n} \ \text{ad}\big( [a_i,b_i]\big) \,  \Big)(A_1) \ = \ \{0\}.$$ There exists a subalgebra $A_2\subset A$ such that $$A_1 \subseteq A_2, \quad  a, \, a_i, \, b_i \in  A_2,  \quad 1 \leq i \leq n, \quad \text{and} \quad A_2 \cong M_m(\mathbb{F}).$$ As above, we consider the centralizer $C$ of the subalgebra $A_1$ in $A_2$ such that $$A_2  =  A_1  \otimes_{\mathbb{F}} C, \quad C \cong M_t(\mathbb{F}), \quad \text{and} \quad t= m / p^k \quad \text{is not a multiple of} \quad p. $$

Consider the element $$b \ = \ \sum_{i=1}^{n} \ [a_i,b_i] \ \in \ A_2. $$ The difference $a-b$ commutes with all elements from $A_1,$ hence $a-b=c\in C.$

In the algebra $A_2$ we have $$ \text{tr}_{A_2}(c) \ = \ \text{tr}_{A_2}(1\otimes c) \ = \ \text{tr}_{A_1}(1) \cdot \text{tr}_{C}(c) \ = \ 0.$$ Hence $$ \text{tr}_{A_2}(a) \ = \ \text{tr}_{A_2}(b) \ + \ \text{tr}_{A_2}(c)  \ = \ 0.$$

On the other hand, $$ \text{tr}_{A_2}(a) \ = \ \text{tr}_{A_2}(a\otimes 1) \ = \ \text{tr}_{A_1}(a) \cdot t \ \neq \ 0.$$ This contradiction completes the proof of Theorem  \ref{theorem3}. \end{proof}

\end{document}